\documentclass[10pt,a4paper,reqno]{amsart}
\usepackage[foot]{amsaddr}

\usepackage{amsmath,amsfonts,amsthm,epsfig,graphicx,mathtools,tikz-cd, amssymb,setspace,enumerate,enumitem}
\usepackage{caption}
\usepackage{mathrsfs,color}
\usepackage[margin=1in]{geometry}
\usepackage{mathabx}

\usepackage{mathtools}
\mathtoolsset{showonlyrefs}
\usepackage[page,header]{appendix}
\usepackage{titletoc}

\usepackage[page,header]{appendix}
\usepackage{titletoc}

\newcommand{\rd}{\,\mathrm{d}}
\numberwithin{equation}{section}
\newtheorem{theorem}{Theorem}[section]
\newtheorem{lemma}[theorem]{Lemma}
\newtheorem{corollary}[theorem]{Corollary}

\newtheorem{definition}[theorem]{Definition}
\newtheorem{remark}[theorem]{Remark}

\usepackage{mathtools}

\def\bu{{\bf u}}

\def\bx{{\bf x}}
\def\by{{\bf y}}

\def\cE{\mathcal{E}}

\def\supp{\textnormal{supp\,}}

\def\dist{\textnormal{dist\,}}

\begin{document}
\title[Pressureless Euler-Poisson equations with quadratic confinement]{Existence of radial global smooth solutions to the pressureless Euler-Poisson equations with quadratic confinement}

\author{Jos\'e A. Carrillo$^{\dagger}$, Ruiwen Shu$^{\ddag}$}
\date{\today}
\subjclass[2020]{}
\address[$\dagger$]{Mathematical Institute, University of Oxford, Oxford OX2 6GG, UK. Email: {\tt carrillo@maths.ox.ac.uk}}
\address[$\ddag$]{Department of Mathematics,
University of Georgia,
Athens, GA 30602, USA.
Email: {\tt ruiwen.shu@uga.edu}}

\maketitle

\begin{abstract}
    We consider the pressureless Euler-Poisson equations with quadratic confinement. For spatial dimension $d\ge 2,\,d\ne 4$, we give a necessary and sufficient condition for the existence of radial global smooth solutions, which is formulated explicitly in terms of the initial data. This condition appears to be much more restrictive than the critical-threshold conditions commonly seen in the study of Euler-type equations. To obtain our results, the key observation is that every characteristic satisfies a periodic ODE system, and the existence of global smooth solution requires the period of every characteristic to be identical.
\end{abstract}

\section{Introduction}

In this work, we will deal with the pressureless Euler-Poisson equations with confinement written as
\begin{equation}\label{eq}\left\{\begin{split}
& \partial_t \rho + \nabla\cdot (\rho \bu) = 0 \\
& \partial_t \bu + \bu\cdot\nabla \bu = - \int \nabla N(\bx-\by)\rho(t,\by)\rd{\by} - \bx \\
\end{split}\right.\,.\end{equation}
Here $\bx\in\mathbb{R}^d,\,d\ge 2$, $\rho(t,\bx)$ is the particle density function, and $\bu(t,\bx)$ is the velocity field. $N$ is the Newtonian repulsion potential, satisfying $-\Delta N = \delta$, given by
\begin{equation}
N(\bx) = \left\{\begin{split}
& -\frac{1}{2\pi}\ln|\bx|,\quad d=2 \\
& c_d|\bx|^{2-d},\quad d\ge 3,\quad c_d = \frac{1}{|S^{d-1}|}
\end{split}\right.\,.
\end{equation}
The last term $-\bx$ in the velocity $\bu$ equation represents the effect of a quadratic confining potential. Notice that this is equivalent to say that the particles are subject to a potential force with the potential being $\phi= (-\Delta)^{-1}(\rho-d)$, i.e., Newtonian repulsion with a positive charged background, see for instance \cite{tan2021eulerian}. Our aim is to give a sharp result on the existence of global smooth solutions to \eqref{eq} for radial initial data.

The existence of global smooth solutions to Euler-Poisson systems has been thoroughly studied in the literature. One popular approach for the study of Eulerian dynamics, which we will adopt in this paper, is \emph{spectral dynamics} \cite{engelberg2001critical,liu2002spectral}. This method was originally designed to analyze the eigenvalues of the deformation matrix $\nabla \bu$ along the characteristics of the flow. It was later generalized to analyze the time evolution of certain quantities along characteristics, and derive the existence of global smooth solutions of the PDE system as that of a family of ODE systems. For the pressureless Euler-Poisson system, some criteria for the existence of global smooth solutions have been developed by \cite{engelberg2001critical,liu2002spectral,LT03,liu2004rotation,bae2012critical,tan2021eulerian} in the context of 1D or multi-D radial solutions. Similar approaches were also developed to study Eulerian dynamics arising from models of collective behavior \cite{tadmor2014critical,carrillo2016critical,CCZ16,he2017global,shvydkoy2017eulerian,do2018global,lear2019existence,CWZ20,shu2020flocking,tan2020euler,tan2021eulerian}, which usually involve other forcing terms like the Cucker-Smale alignment interaction \cite{cucker2007emergent} or linear damping. 

The local-in-time existence and uniqueness of classical solutions to the Euler-Poisson system is known for the initial data being a small perturbation of the stationary state, see \cite{M86,MP90}. In these references, the authors assume that the density is positive on the whole line with zero limit as $x\to \pm\infty$. 
A local-in-time well-posedness of the Cauchy problem for the pressureless Euler-Poisson system in the plane without smallness assumptions in Sobolev spaces was given in \cite[Section 5]{bae2012critical}. Besides the study of the pressureless Euler-Poisson system, people have also studied the existence of global smooth solutions to the Euler-poisson system with pressure \cite{wang1998formation,guo1998smooth,wang2001global,tadmor2008global,yuen2011blowup,jang2012two}. We can summarize by saying that finding sharp criteria for the existence of global smooth solutions is a challenging problem for Euler-Poisson type problems. One of the difficulties we need to face in this work is to deal with initial data that are compactly supported in the density for \eqref{eq}, and thus we need to introduce a suitable notion of solution consistent with  free boundary conditions for the system \eqref{eq}.

\subsection{Radial formulation \& Notion of solution}

As already mentioned, we are concerned with radial solutions to \eqref{eq}, i.e., solutions with $\rho=\rho(t,r)$, $\bu=u(t,r)\frac{\bx}{r}$, where $r=|\bx|$. To reformulate \eqref{eq} into radial variables, we introduce the quantities related to a density $\rho(\bx)=\rho(r)$:
\begin{equation}
P (r) = |S^{d-1}|r^{d-1}\rho(r),\quad m(r) = \int_0^r  P (s)\rd{s} = \int_{|\by|<r} \rho(\by)\rd{\by}.
\end{equation}
Similar notations will be used for time-dependent densities. We  give a lemma on the Newtonian potential generated by a radial density.
\begin{lemma}
Let $\rho(\bx)=\rho(r)$ be compactly supported and $L^\infty$. Then
\begin{equation}\begin{split}
\int N(\bx-\by)\rho(|\by|)\rd{\by} = \int_r^\infty N(s) P (s)\rd{s} + N(r)m(r).
\end{split}\end{equation} 
\end{lemma}

\begin{proof}
Denote $R=|\bx|$. We have
\begin{equation}\begin{split}
\int N(\bx-\by)\rho(|\by|)\rd{\by} = & \int_{|\by|>R}N(\bx-\by)\rho(|\by|)\rd{\by} +  \int_{|\by|<R}N(\bx-\by)\rho(|\by|)\rd{\by}  \\
= & \int_{|\by|>R}N(\by)\rho(|\by|)\rd{\by} + N(\bx)\int_{|\by|<R}\rho(|\by|)\rd{\by}.
\end{split}\end{equation} 
Here we treat the first integral by the fact that $\int_{|\by|>R}N(\cdot-\by)\rho(|\by|)\rd{\by} $ is a radial harmonic function on $B(0;R)$ and continuous on $\mathbb{R}^d$, and thus constant on $\widebar{B(0;R)}$. We use the mean-value property of the harmonic function $N$ on $B(\bx;|\by|)$ in the second integral. Therefore the conclusion is obtained.
\end{proof}

Now we can write radial solutions to \eqref{eq} as 
\begin{equation}\label{eqrad}\left\{\begin{split}
& \partial_t P + \partial_r (P u) = 0 \\
& \partial_t u + u\partial_r u = - \partial_r N(r)m(t,r) - r,\quad m(t,r) = \int_0^r P(t,s)\rd{s}
\end{split}\right. .\end{equation}
We always assume that the radial initial data $(\rho_0,\bu_0)$ of \eqref{eq} satisfies that $\rho_0$ is continuous and compactly supported with $\rho_0\ge 0$, $\rho_0(0)>0$, and $\bu_0$ is $C^1$ on $\supp\rho_0$. As a consequence, the corresponding initial data $(P_0,u_0)$ of \eqref{eqrad} satisfies that 
\begin{itemize}
    \item $P_0$ is $C^1$, compactly supported on $[0,R_0]$ for some $R_0>0$, with  $\lim_{r\rightarrow 0^+}r^{1-d}P_0(r)>0$ and $\partial_r(r^{1-d}P_0(r))|_{r=0}=0$.
    \item $u_0$ is $C^1$ on $[0,R_0]$ with $u_0(0)=0$.
\end{itemize}
The triple $(P_0(r),u_0(r),R_0)$ is said to be \emph{consistent} if the above two conditions are satisfied.

\begin{definition}
A tuple $(P(t,r),u(t,r),R(t))$ is called a \emph{classical bulk solution} to \eqref{eqrad} on $[0,T],T>0$ with the consistent initial data $(P_0,u_0,R_0)$  if
\begin{itemize}
\item $R\in C^1([0,\infty))$. $P$ and $u$ are supported on $\{(t,r):0\le t \le T,\,r\in [0,R(t)]\}$ and $C^1$ on this set. $P$, $u$, $R$ agree with the initial data at $t=0$.
\item \eqref{eqrad} is satisfied in $\{(t,r):0\le t \le T,\,r\in (0,R(t))\}$ in the classical sense.
\item $R(t)$ correctly describes the boundary motion, i.e., $R'(t)=u(t,R(t))$.
\end{itemize}
It is called a \emph{global classical bulk solution} if it is a classical bulk solution on $[0,T]$ for any $T>0$.
\end{definition}

We will provide a continuation criteria for bulk solutions to \eqref{eqrad} in Section 2, see Lemma \ref{lem_crit}. This allows us to analyse the global existence of bulk solutions by characteristic tracing. 

\subsection{Main results}

With these preparations we can state the main contributions of this work. Because of the specificity of the two dimensional Newtonian potential, we separate the general result from the two dimensional case.

\begin{theorem}\label{thm1}
Assume $d\ge 3$, $d\ne 4$. Let $(P_0,u_0,R_0)$ be a consistent initial data. Then there exists a global classical bulk solution to \eqref{eqrad} with this initial data if and only if the following conditions are satisfied:
\begin{itemize}
\item There exists a constant $C_0$ such that 
\begin{equation}\label{thm1_0}
    m_0(r)^{-2/d}\left(\frac{1}{2}u_0(r)^2 + m_0(r) N(r) + \frac{1}{2}r^2\right)=C_0,\quad \forall r\in (0,R_0).
\end{equation}
\item Either $C_0 = \min_{\tilde{r}}\left\{ N(\tilde{r})+\frac{1}{2}\tilde{r}^2 \right\}$ and $(P_0,u_0,R_0)$ is a stationary solution; or $C_0 > \min_{\tilde{r}}\left\{ N(\tilde{r})+\frac{1}{2}\tilde{r}^2 \right\}$ and
\begin{equation}\label{thm1_1}
    \min_{(\tilde{r},\tilde{u})\in \mathcal{K}_r} \left\{\theta(r)\tilde{u} + \frac{P_0(r)}{d m_0(r)} \tilde{r} \right\} > 0,\quad \forall r\in (0,R_0)
\end{equation}
where the minimum is taken over the energy level set 
$$
\mathcal{K}_r:=\left\{(\tilde{r},\tilde{u})\in\mathbb{R}_+\times\mathbb{R}:m_0(r)^{-2/d}\left(\frac{1}{2}\tilde{u}^2 + m_0(r) N(\tilde{r}) +  \frac{1}{2}\tilde{r}^2\right)=C_0\right\},
$$
and $\theta(r)$ is defined by
\begin{equation}\label{theta}
    \theta(r) := \left\{\hspace{-2mm}
    \begin{array}{cl}
\displaystyle  \frac{1-\frac{P_0(r)r}{d m_0(r)}}{u_0(r)}   & \mbox{if } u_0(r)\neq 0 \\[5mm]
\displaystyle      \frac{\frac{1}{d}P_0(r)u_0(r)-m_0(r) \partial_r u_0(r)}{m_0(r)(-c_d m_0(r) (d-2)r^{1-d}+r)}   & \mbox{if } -c_d m_0(r) (d-2)r^{1-d}+r\neq 0 
    \end{array}
     \right. \,.
\end{equation}
Here item 1 and $C_0 > \min_{\tilde{r}}\left\{ N(\tilde{r})+\frac{1}{2}\tilde{r}^2\right\}$ guarantee that at least one of the above fractions have nonzero denominator, and they are equal when both having nonzero denominators.
\end{itemize}
\end{theorem}

It is easy to see that item 1 implies the equivalence of the two fractions in \eqref{theta}. In fact, differentiating the energy level equation  \eqref{thm1_0} with respect to $r$ gives
\begin{equation}
-\frac{2}{d}P_0(r)\Big(\frac{1}{2}u_0(r)^2+m_0(r) N(r) + \frac{1}{2}r^2\Big) + m_0(r)\Big( u_0(r) \partial_r u_0(r) + P_0(r) N(r) + m_0(r) \partial_r N(r) + r\Big)=0.
\end{equation}
Using $N(r) = c_d r^{2-d}$, one can rewrite it as 
\begin{equation}\label{C0d}
    u_0(r)\Big(-\frac{1}{d}P_0(r)u_0(r)+m_0(r) \partial_r u_0(r)\Big) + m_0(r)\Big(1-\frac{P_0(r)r}{d m_0(r)}\Big)\Big(-c_d m_0(r) (d-2)r^{1-d}+r\Big) = 0
\end{equation}
which shows that the two fractions in \eqref{theta} are equal whenever they have nonzero denominators. Also, if both denominators are zero, then one has $u_0(r)=-c_d m_0(r) (d-2)r^{1-d}+r=0$, in which case we will show that the same is true for any $0<r<R_0$ and we necessarily have a stationary solution to \eqref{eqrad}.

\begin{theorem}\label{thm1_2d}
Assume $d=2$. Let $(P_0,u_0,R_0)$ be a consistent initial data. Then there exists a global classical bulk solution to \eqref{eqrad} with this initial data if and only if the following conditions are satisfied:
\begin{itemize}
\item There exists a constant $C_0$ such that \begin{equation}\label{thm1_2d_0}
    m_0(r)^{-1}\Big(\frac{1}{2}u_0(r)^2 + m_0(r) N(r) + \frac{1}{2}r^2\Big) = C_0-\frac{1}{4\pi}\ln m_0(r),\quad \forall r\in (0,R_0).
\end{equation}
\item Either $C_0 = \min_{\tilde{r}} \{N(\tilde{r})+\frac{1}{2}\tilde{r}^2\}$ and $(P_0,u_0,R_0)$ is a stationary solution; or $C_0 > \min_{\tilde{r}} \{N(\tilde{r})+\frac{1}{2}\tilde{r}^2\}$ and
\begin{equation}\label{thm1_2d_1}
    \min_{(\tilde{r},\tilde{u})\in\mathcal{K}_r}\left\{ \theta(r)\tilde{u} + \frac{P_0(r)}{2 m_0(r)} \tilde{r}\right\} > 0,\quad \forall r\in (0,R_0),
\end{equation}
where the minimum is taken over the energy level set 
$$
\mathcal{K}_r:=\left\{(\tilde{r},\tilde{u})\in\mathbb{R}_+\times\mathbb{R}:m_0(r)^{-1}\Big(\frac{1}{2}\tilde{u}^2 + m_0(r) N(\tilde{r}) + \frac{1}{2}\tilde{r}^2\Big)=C_0-\frac{1}{4\pi}\ln m_0(r)\right\},
$$ 
and $\theta(r)$ is defined by
\begin{equation}\label{theta_2d}
    \theta(r) := \left\{\hspace{-2mm}
    \begin{array}{cl}
\displaystyle  \frac{1-\frac{P_0(r)r}{2 m_0(r)}}{u_0(r)}   & \mbox{if } u_0(r)\neq 0 \\[5mm]
\displaystyle      \frac{\frac{1}{2}P_0(r)u_0(r)-m_0(r) \partial_r u_0(r)}{m_0(r)(-\frac{1}{2\pi} m_0(r) r^{-1}+r)}   & \mbox{if } -\frac{1}{2\pi} m_0(r) r^{-1}+r\neq 0 
    \end{array}
      \right. \,. 
\end{equation}
Here item 1 and $C_0 > \min_{\tilde{r}}\{ N(\tilde{r})+\frac{1}{2}\tilde{r}^2\}$ guarantee that at least one of the above fractions have nonzero denominator, and they are equal when both having nonzero denominators.
\end{itemize}
\end{theorem}

Similarly, differentiating the energy level equation  \eqref{thm1_2d_0} with respect to $r$ gives
\begin{equation}\begin{split}
-P_0(r) & \Big(\frac{1}{2}u_0(r)^2+m_0(r) N(r) + \frac{1}{2}r^2\Big) \\
& + m_0(r)\Big( u_0(r) \partial_r u_0(r) + P_0(r) N(r) + m_0(r) \partial_r N(r) + r\Big)=-\frac{1}{4\pi}m_0(r)P_0(r)\,.
\end{split}\end{equation}
Using $N(r) = -\frac{1}{2\pi}\ln r$, one can rewrite it as 
\begin{equation}\label{C0d_2d}
    u_0(r)\Big(-\frac{1}{2}P_0(r)u_0(r)+m_0(r) \partial_r u_0(r)\Big) + m_0(r)\Big(1-\frac{P_0(r)r}{2 m_0(r)}\Big)\Big(-\frac{1}{2\pi} m_0(r) r^{-1}+r\Big) = 0
\end{equation}
which shows that the two fractions in \eqref{theta_2d} are equal whenever they have nonzero denominators. The case where both denominators are zero also correspond to stationary solutions.

\subsection{Sketch of the proof \& Plan of the paper}

The proof of Theorems \ref{thm1} and \ref{thm1_2d} is based on tracing the characteristics. In fact, along each characteristic of \eqref{eqrad}, the quantities $r$ and $u$ satisfies a closed ODE system \eqref{eqru}. Each ODE system is a one-dimensional Hamiltonian system whose solution is necessarily periodic. The key observation is that the existence of global classical bulk solution requires that all these ODE systems necessarily have \emph{the same period} (Lemma \ref{lem_Tconst}). In fact, by some elementary argument, it is not hard to show that if nearby characteristics have various periods, then they will intersect at some time which breaks the classical solution to \eqref{eqrad}. 

In Section \ref{sec_per} we will first prove that for $d\ge 2,\,d\ne 4$ the energy level of every characteristic (with suitable rescaling) has to be the same, leading to the condition \eqref{thm1_0} (or its 2D counterpart \eqref{thm1_2d_0}). This is a consequence of the fact that the period $T(E)$, as a function of the energy level $E$, is \emph{non-constant} on any interval (Lemma \ref{lem_Tnon}, see Figure \ref{fig0} as an illustration), and the energy level changes continuously among characteristics. The proof of Lemma \ref{lem_Tnon} is based on the \emph{real analytic} property of $T(E)$ which will be established independently in Section \ref{sec_ana}. 

\begin{figure}[ht]
    \centering
    \includegraphics[width=0.8\textwidth]{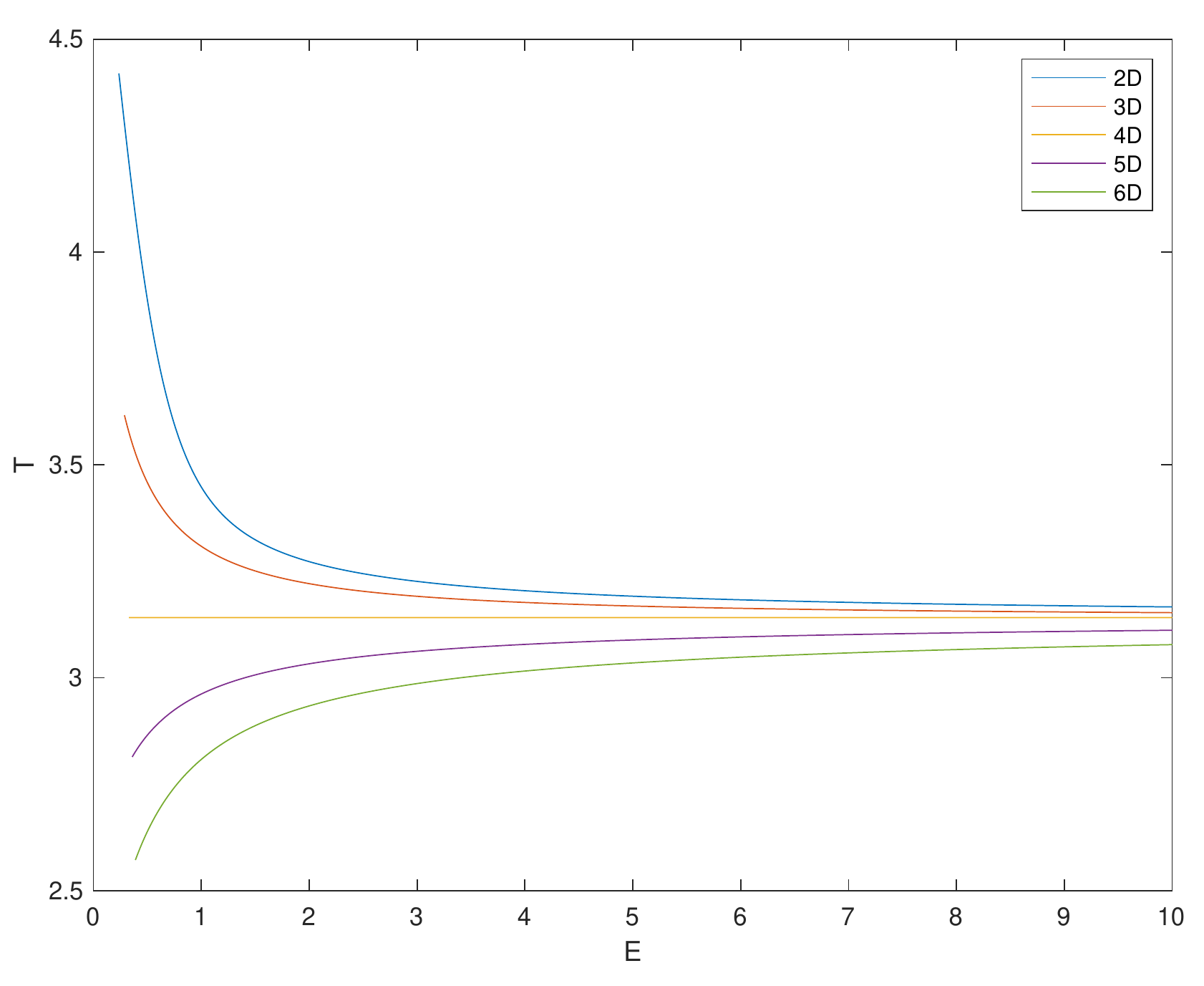}
    \caption{Period $T$ of the ODE system \eqref{eqru} (with $m=1$) as a function of the energy level $E$. From the top curve to the bottom curve,  we take the spatial dimension for $N(r)$ as $2,3,4,5,6$. Corollary \ref{cor_Vd} below proves that $T(E)$ is constant for 4D but non-constant for other dimensions.}
    \label{fig0}
\end{figure}

Once the necessity of the energy level condition \eqref{thm1_0} (or \eqref{thm1_2d_0} in 2D) is proved, we may assume this condition and study the possible blow-up phenomenon within one period. We do this in Section \ref{sec_crit} by analyzing the ODE system \eqref{eqrad_d} satisfied by $P$ and $w:=\partial_r u$ along characteristics, and obtain the critical-threshold-type condition \eqref{thm1_1} (or \eqref{thm1_2d_1} in 2D).

\begin{remark}
To the authors' best knowledge, equality conditions like \eqref{thm1_0} have not been observed in existing literature for global classical solutions to pressureless Euler-Poisson / Euler-Poisson-alignment systems. Typically one obtains a critical-threshold-type condition like \eqref{thm1_1} as a necessary or necessary-and-sufficient condition for the existence of global classical solutions. We explain below why a condition like \eqref{thm1_0} is not needed in these references. 
\begin{itemize}
    \item For the 1D pressureless Euler-Poisson system with quadratic confinement, it is well-known that every characteristic is a harmonic oscillator with the same period \cite[Theorem 3.1]{engelberg2001critical}, and one does not require that every characteristic has the same energy level. We expect similar situation in 4D (c.f. Corollary \ref{cor_Vd}) which is left as a future work.
    \item For pressureless (repulsive) Euler-Poisson system without confinement, if one assumes radial symmetry, then all characteristics escape to infinity and thus not periodic \cite{bae2012critical}. Therefore our argument on the periods of characteristics does not apply.
    \item For Euler-alignment systems, the alignment effect is energy-dissipative and thus stablizes the characteristic ODE systems. Therefore one expects to have global solution converging to equilibrium as long as the initial condition is close to equilibrium (in certain norms) or the alignment effect is sufficiently strong \cite{CCZ16,tan2021eulerian}. Therefore an equality condition like \eqref{thm1_0} is not expected as a necessary condition.
\end{itemize}
As a summary, our system \eqref{eqrad} is energy-conservative and every characteristic is periodic. The period of every characteristic has to be the same, which requires the energy level of every characteristic being identical. This explains why a condition like \eqref{thm1_0} is necessary.
\end{remark}

\section{Proof of Theorem \ref{thm1} by characteristic tracing }

We start by providing a local-in-time existence theory and continuation criterion for bulk solutions to \eqref{eqrad}. The following continuation criterion reduces the existence problem of global classical bulk solution to the boundedness of $r^{1-d}P(t,r)$ and $\partial_r u(t,r)$. 

\begin{lemma}\label{lem_crit}
Assume $d\ge 2$, and let $(P_0,u_0,R_0)$ be a consistent initial data. Then there exists $T=T(C_0)>0,\,C_0:=\|r^{1-d}P_0(r)\|_{L^\infty(0,R_0)}+ \|\partial_r u_0\|_{L^\infty(0,R_0)}$ such that there exists a unique classical bulk solution to \eqref{eqrad} on $[0,T]$ with this initial data.

As a consequence, if $0<T<\infty$ is the maximal time of existence of a classical bulk solution to \eqref{eqrad}, then one necessarily has $\lim_{t\rightarrow T^-} \|r^{1-d}P(t,r)\|_{L^\infty_r(0,R(t))}+ \|\partial_r u(t,r)\|_{L^\infty_r(0,R(t))} = \infty$.
\end{lemma}

\begin{proof}

Let $\Phi(t;r),\,t\ge 0,0<r\le R_0$ denote the characteristic flow, i.e., the value of $\tilde{r}(t)$ of the solution to the ODE system 
\begin{equation}\label{eqru1}\left\{\begin{split}
& \tilde{r}' = \tilde{u} \\
& \tilde{u}' = -m_0(r) \partial_r N(\tilde{r}) - \tilde{r}
\end{split}\right.,\quad \tilde{r}(0) = r,\,\tilde{u}(0) = u_0(r)\,.\end{equation}
Recall that $\partial_r N(\tilde{r}) = -c_d (d-2) \tilde{r}^{1-d}, \partial_{rr} N(\tilde{r}) = c_d (d-2)(d-1) \tilde{r}^{-d}$ (for $d\ge 3$; $d=2$ is similar). We may rewrite \eqref{eqru1} as 
\begin{equation}\label{eqru1_1}\left\{\begin{split}
& (\tilde{r}/r)' = \tilde{u}/r \\
& (\tilde{u}/r)' = -m_0(r)r^{-d} \partial_r N(\tilde{r}/r) - \tilde{r}/r
\end{split}\right.,\quad (\tilde{r}/r)(0) = 1,\,(\tilde{u}/r)(0) = u_0(r)/r\,.\end{equation}
Using the condition $u_0(0)=0$ for consistent initial data, the above initial data and RHS coefficient $m_0(r)r^{-d}$ are also controlled by $C C_0$. This implies the existence of $T=T(C_0)>0$ such that
\begin{equation}\label{aux}
    (\tilde{r}/r)(t)\in (1/2,2),\quad \mbox{and} \quad |(\tilde{u}/r)(t)| \le 2C_0 \quad \mbox{for } 0<t\le T\,.
\end{equation}

By viewing $\tilde{r},\tilde{u}$ in \eqref{eqru1} as functions of $t$ and $r$ and differentiating with respect to $r$, we get
\begin{equation}\label{eqru2}\left\{\begin{split}
& \partial_r \tilde{r}' = \partial_r \tilde{u} \\
& \partial_r \tilde{u}' = -P_0(r) \partial_r N(\tilde{r}) - m_0(r) \partial_{rr}N(\tilde{r}) \partial_r \tilde{r} - \partial_r \tilde{r}
\end{split}\right.,\quad \partial_r \tilde{r}(0) = 1,\,\partial_r \tilde{u}(0) = \partial_r u_0(r)\,.\end{equation}
For $0<t\le T$, we have the estimate for the coefficients
\begin{equation}
    |P_0(r) \partial_r N(\tilde{r})| \le C P_0(r) r^{1-d} (\tilde{r}/r)^{1-d} \le C C_0,\quad |m_0(r) \partial_{rr}N(\tilde{r})| \le C C_0\,,
\end{equation}
by using \eqref{aux}.
Therefore, by making $T=T(C_0)$ smaller if necessary, we may guarantee that $\partial_r \tilde{r}(t;r) > 1/2 $ for any $0<t\le T$ and  $0< r \le R_0$, i.e., $\partial_r \Phi(t;r) > 1/2$. Therefore, for any $0<t\le T$, $\Phi(t;\cdot)$ is an invertible map from $(0,R_0]$ to $(0,R(t)],\,R(t) := \Phi(t;R_0)$. Denoting its inverse as $\Psi(t;\cdot): (0,R(t)]\rightarrow (0,R_0]$, t hen $\Psi(t;\cdot)$ is $C^1$ with $\|\partial_r \Psi(t;r)\|_{L^\infty_r(0,R(t))} \le 2$ for any $0<t\le T$. 

Then we may define the solution $(P,u,R)$ on $[0,T]$ by tracing back the characteristics:
\begin{equation}
    P(t,r) = P_0(\Psi(t;r)),\quad u(t,r) = u_0(\Psi(t;r)),\quad 0<r\le R(t)\,.
\end{equation}
It is straightforward to verify that $(P,u,R)$ is a classical bulk solution to \eqref{eqrad} on $[0,T]$ with the desired initial data. The uniqueness of classical bulk solution follows from the uniqueness of the solution to the characteristic ODE \eqref{eqru1}.

\end{proof}

Now, let us consider a classical bulk solution to \eqref{eqrad} and discuss its global-in-time existence. Denote $' = \partial_t + u\partial_r$ as the derivative along characteristics. Then it is clear that $m'=0$, i.e., $m$ is constant along characteristics. This constant value of $m$ is positive as long as the characteristics starts in $(0,R_0)$, by the assumption $\lim_{r\rightarrow 0^+}r^{1-d}P_0(r)>0$ (which corresponding to $\rho_0(0)>0$ for \eqref{eq}). The evolution of $r$ and $u$ along characteristics is given by
\begin{equation}\label{eqru}\left\{\begin{split}
& r' = u \\
& u' = -m \partial_r N(r) - r
\end{split}\right.\end{equation}
which is a closed ODE system along characteristics (depending on the constant value $m$ along these characteristics). This system has the particle energy as a conserved quantity:
\begin{equation}
\cE(r,u;m) = \frac{1}{2}u^2 + m N(r) + \frac{1}{2}r^2.
\end{equation}
Since \eqref{eqru} is two-dimensional and the energy $\cE$ is convex and coercive on $(r,u)\in (0,\infty)\times \mathbb{R}$ for any fixed $m>0$, we see that any solution to \eqref{eqru} is necessarily periodic (possibly degenerate to an equilibrium point), and the orbit is a level set of $\cE$, uniquely determined by its energy level.

\subsection{Analysis of the period}\label{sec_per}

We denote $T(\cE_0;m)$ as the period of the orbit for \eqref{eqru} with energy level $\cE_0$, which is defined for any $\cE_0>\cE_{\min}(m):=\min_r \{m N(r) + \frac{1}{2}r^2\}$. We may also define $T(\cE_{\min}(m);m)$ by taking the limit, and make $T(\cdot;m)$ continuous on $[\cE_{\min}(m),\infty)$. In fact, one can justify the existence of the limit $\lim_{\cE\rightarrow \cE_{\min}(m)^+}T(\cE;m)$ by linearizing the dynamics of $\eqref{eqru}$ near the equilibrium point. As a function of $\cE_0$ and $m$, $T$ is clearly continuous on $\{(m,\cE_0): m>0,\,\cE_0\in [\cE_{\min}(m),\infty)\}$.

Furthermore, linearization gives the value of $T(\cE_{\min}(m);m)$ as 
\begin{equation}\label{Tmin}
T(\cE_{\min}(m);m) = 2\pi\Big(\Big(\frac{\rd^2}{\rd{r}^2}\big(m N(r) + \frac{1}{2}r^2\big)\Big)\Big|_{r=\text{argmin} \{m N(r) + \frac{1}{2}r^2\}}\Big)^{-1/2} =: \tau_d\,,
\end{equation}
by approximating the orbits near the equilibrium by the right harmonic oscillator.
Using the definition $N(r) = c_d r^{2-d}$, it is easy to verify that this expression is independent of $m$, and thus we may denote it as $\tau_d$. 

For a given consistent initial data $(P_0,u_0,R_0)$, denote
\begin{equation}
\cE_0(r) = \frac{1}{2}u_0(r)^2 + m_0(r) N(r) + \frac{1}{2}r^2,\quad r\in (0,R_0)
\end{equation}
as the initial energy level of each characteristic.
\begin{lemma}\label{lem_Tconst}
If there exists a global classical bulk solution to \eqref{eqrad} with consistent initial data $(P_0,u_0,R_0)$, then $T(\cE_0(r);m_0(r))$ is constant in $r\in (0,R_0)$.
\end{lemma}

\begin{proof}
$(P_0,u_0,R_0)$ are assumed to be $C^1$, and thus $\cE_0(r)$ and $T(\cE_0(r);m_0(r))$ are continuous on $(0,R_0)$, and $C^1$ at any $r\in (0,R_0)$ with $\cE_0(r) > \cE_{\min}(m_0(r))$. 

We first claim that for any $r\in (0,R_0)$, either $\frac{\rd}{\rd{r}}(T(\cE_0(r);m_0(r)))=0$, or $\cE_0(r) = \cE_{\min}(m_0(r))$.

Suppose on the contrary that $\frac{\rd}{\rd{r}}(T(\cE_0(r);m_0(r)))|_{r=r_*}\ne 0$ for some $r_*\in (0,R_0)$ with $\cE_0(r_*) > \cE_{\min}(m_0(r_*))$. Then applying the inverse function theorem, we see that for $\epsilon\in\mathbb{R}\backslash\{0\}$ with sufficiently small absolute value, there exists a unique $r_\epsilon$ near $r_*$ such that
\begin{equation}
T(\cE_0(r_\epsilon);m_0(r_\epsilon)) = T(\cE_0(r_*);m_0(r_*)) +\epsilon
\end{equation}
and $r_\epsilon\rightarrow r_*$ as $\epsilon\rightarrow 0$. $r_\epsilon-r_*$ has the same sign as $\frac{\rd}{\rd{r}}(T(\cE_0(r);m_0(r)))|_{r=r_*}\cdot \epsilon$. 

\begin{figure}[ht]
    \centering
    \includegraphics[width=0.99\textwidth]{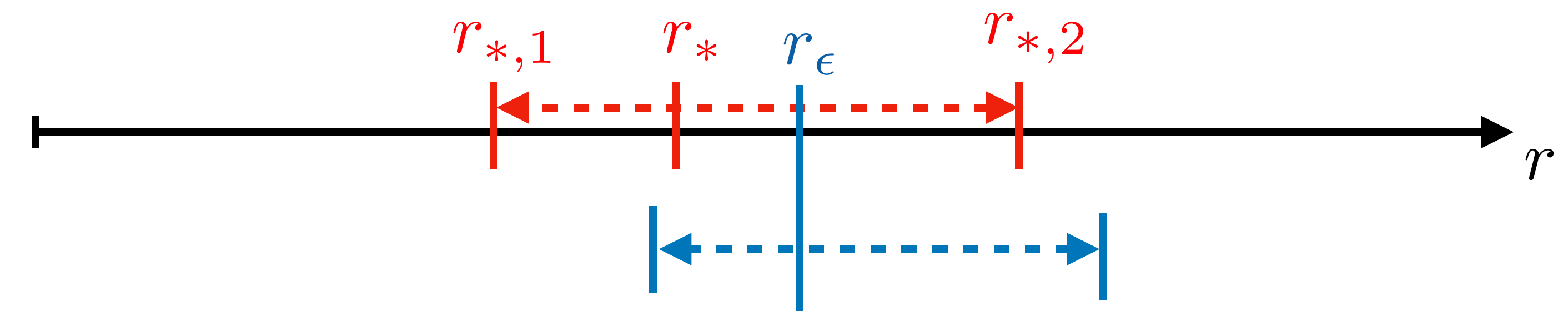}
    \caption{Proof of Lemma \ref{lem_Tconst}. The red interval is the range in which the characteristic starting from $r_*$ oscillates. If $T$ is not constant, we may always find $r_\epsilon$ nearby, for which the two characteristics starting from $r_*$ and $r_\epsilon$ intersect.}
    \label{fig1}
\end{figure}

Since $\cE_0(r_*) > \cE_{\min}(m_0(r_*))$, the orbit starting from $r_*$ is non-trivial, and its $r$-variable travels between $r_{*,1}$ and $r_{*,2}$, which are the two solutions to the equation $\cE_0(r_*) = m_0(r) N(r) + \frac{1}{2}r^2$ in $r$. They satisfy $r_{*,1}\le r_*\le r_{*,2}$ and $r_{*,1}<r_{*,2}$. 

Without loss of generality, we assume $r_{*,2}>r_*$. Then we take $\epsilon$ to have the same sign as 
$$
\frac{\rd}{\rd{r}}(T(\cE_0(r);m_0(r)))|_{r=r_*},
$$ 
sufficiently small, and having irrational ratio with $T(\cE_0(r_*);m_0(r_*))$. This guarantees $r_\epsilon\in (r_*,r_{*,2})$, and thus the ranges of the orbits starting from $r_*$ and $r_\epsilon$ intersect. Then, by a standard rational approximation argument, one can show that these two orbits intersect at some time $t$, which contradicts the existence of global classical bulk solution. This proves the claim. See Figure \ref{fig1} for illustration.

Due to \eqref{Tmin}, the conclusion of the lemma would be trivial if $\cE_0(r) = \cE_{\min}(m_0(r))$ for any $r\in (0,R_0)$. On the other hand, if there exists some $r_*\in (0,R_0)$ with $\cE_0(r_*) > \cE_{\min}(m_0(r_*))$, then by continuity we may take a maximal open interval $(r_1,r_2)\subset (0,R_0)$ containing $r_*$ with the same property. Then the claim gives that $T(\cE_0(r);m_0(r))$ is constant on $(r_1,r_2)$. If $(r_1,r_2)=(0,R_0)$, then the conclusion follows; otherwise, if we assume without loss of generality that $r_1>0$, then $\cE_0(r_1) = \cE_{\min}(m_0(r_1))$, which gives $T(\cE_0(r_1);m_0(r_1))=\tau_d$ by \eqref{Tmin}. Then $T(\cE_0(r);m_0(r))=\tau_d$ on $(r_1,r_2)$ by continuity. Then applying this for every $r$ with with $\cE_0(r) > \cE_{\min}(m_0(r))$, we see that $T(\cE_0(r);m_0(r))=\tau_d$ for every $r\in (0,R_0)$, which implies the conclusion.
\end{proof}

\begin{remark}
Following the idea of this proof, one can see that there is no global classical solution if one requires $P$ and $u$ to be defined on the whole space, provided that $P_0$ is compactly supported. In fact, let $[0,R]$ be an interval containing $\supp P_0$, then $m_0(r)$ is constant (denoted as $M_0$) for $r>R$. Then every characteristic starting from $r>R$ solves the ODE system \eqref{eqru} with $m$ replaced by the constant $M_0$. One clearly has $\lim_{r\rightarrow\infty} \cE_0(r) = \infty$. Therefore, if $[r_1(r),r_2(r)]$ is the interval where the characteristic starting from $r$ travels, we have $\lim_{r\rightarrow\infty}r_1(r)=0,\,\lim_{r\rightarrow\infty}r_2(r)=\infty$. This implies $[r_1(r),r_2(r)]$ with a sufficiently large $r$ contains a similar interval with a smaller $r$. It is clear that such two characteristics necessarily intersect, which breaks down the classical solution. This is the main reason why we introduced the concept of bulk solution and its analysis.
\end{remark}

Next we state the basic properties of $T(\cE_0;m)$.

\begin{lemma}\label{lem_Tscale}
$T(\cE_0;m)$ satisfies the scaling law
\begin{equation}\label{lem_Tscale_1}
T(m^{2/d}\cE_0;m) = T(\cE_0;1),\quad d\ge 3
\end{equation}
and
\begin{equation}\label{lem_Tscale_2}
T\Big(m\Big(\cE_0-\frac{1}{4\pi}\ln m\Big);m\Big) = T(\cE_0;1),\quad d=2.
\end{equation}
\end{lemma}

\begin{proof}
Notice that $\partial_r N(r) = -c r^{1-d}$ for any $d\ge 2$ (with $c>0$ depending on $d$). Therefore, if $(\bar{r},\bar{u})$ solves 
\begin{equation}\label{eqru_n}\left\{\begin{split}
& \bar{r}' = \bar{u} \\
& \bar{u}' = - \partial_r N(\bar{r}) - \bar{r}
\end{split}\right.\end{equation}
then $(m^{1/d} \bar{r}, m^{1/d} \bar{u})$ solves \eqref{eqru}, and has the same period. Notice that $N(r) = \frac{c}{d-2} r^{2-d}$ for $d\ge 3$, and thus
\begin{equation}
\cE(m^{1/d} \bar{r}, m^{1/d} \bar{u};m) = \frac{1}{2}m^{2/d}\bar{u}^2 + m\cdot m^{(2-d)/d}N(\bar{r}) + \frac{1}{2}m^{2/d}\bar{r}^2 = m^{2/d}\cE(\bar{r},\bar{u};1)
\end{equation}
therefore \eqref{lem_Tscale_1} follows. For $d=2$,
\begin{equation}
\cE(m^{1/d} \bar{r}, m^{1/d} \bar{u};m) = \frac{1}{2}m \bar{u}^2 + m \Big(N(\bar{r})-\frac{1}{4\pi}\ln m\Big) + \frac{1}{2}m \bar{r}^2 = m\Big(\cE(\bar{r},\bar{u};1)-\frac{1}{4\pi}\ln m\Big)
\end{equation}
and \eqref{lem_Tscale_2} follows.

\end{proof}

\begin{lemma}\label{lem_Tnon}
If $d\ge 2$, $d\ne 4$, then $T(\cdot;1)$ is non-constant on any sub-interval of $[\cE_{\min}(1),\infty)$.
\end{lemma}
This is a consequence of Corollary \ref{cor_Vd} which will be independently proved.

\begin{proof}[Proof of Theorems \ref{thm1} and \ref{thm1_2d}, necessity of item 1]

Assume there exists a solution to \eqref{eqrad} as stated in Theorem \ref{thm1} or Theorem \ref{thm1_2d}. By Lemma \ref{lem_Tconst}, $T(\cE_0(r);m_0(r))$ is constant in $r\in (0,R_0)$. If $d\ge 3$, then $T(\cE_0(r);m_0(r))=T(m_0(r)^{-2/d}\cE_0(r);1)$ is constant in $r$, and then, Lemma \ref{lem_Tnon} shows that  $m_0(r)^{-2/d}\cE_0(r)$ is constant since it is continuous and $d\ne 4$. This gives the necessity of item 1 in Theorem \ref{thm1}. The $d=2$ case (Theorem \ref{thm1_2d}) can be treated similarly.

\end{proof}

\subsection{Critical threshold for blow-up within a period}\label{sec_crit}

In this subsection we finish the proof of Theorems \ref{thm1} and \ref{thm1_2d} by analyzing the critical thresholds for finite-time blow-up phenomena.

\begin{proof}[Proof of Theorem \ref{thm1}]

In this proof we treat the case $d\ge 3$ and thus prove Theorem \ref{thm1}.

From now on, we assume that item 1 of Theorem \ref{thm1} is satisfied. If $C_0 = \min_{\tilde{r}}\{ N(\tilde{r})+\frac{1}{2}\tilde{r}^2\}$, then \eqref{thm1_0} implies that $u_0(r)=0$ and $r=\text{argmin}\{m_0(r)N(\tilde{r})+\frac{1}{2}\tilde{r}^2\}$ for any $r\in (0,R_0)$, and thus the characteristic dynamics \eqref{eqru} is stationary. This implies that $(P_0,u_0,R_0)$ is a stationary solution.

Otherwise, if $C_0 > \min_{\tilde{r}} \{N(\tilde{r})+\frac{1}{2}\tilde{r}^2\}$, then every characteristic \eqref{eqru} of the solution is non-stationary and has the same period $T_0=T(C_0;1)$, and thus the solution is also $T_0$-periodic if it is global. Therefore the solution is global if and only if it has no blow-up within one period. To study the condition for such blow-up, we differentiate the $u$ equation in \eqref{eqrad} with respect to $r$, combine with the $P$ equation in \eqref{eqrad} and obtain the evolution along characteristics
\begin{equation}\label{eqrad_d}\left\{\begin{split}
& P' = -P w \\
& w' = -w^2 - \partial_{rr} N(r)m - \partial_r N(r)P - 1
\end{split}\right..\end{equation}
Here $w:=\partial_r u$, $m$ is constant along characteristics, and $(r,u)$ satisfies the ODE system \eqref{eqru} which has period $T_0$. We know that \eqref{eqrad_d} necessarily has period $T_0$ as long as it has a global solution.

\eqref{C0d} shows that for every characteristic, the initial condition for the ODE system  \eqref{eqru}\eqref{eqrad_d} for $(r,u,P,w)$ satisfies
\begin{equation}\label{C0d0}
    u(0)\Big(-\frac{1}{d}P(0)u(0)+m w(0)\Big) + m\Big(1-\frac{P(0)r(0)}{d m}\Big)\Big(-c_d m (d-2)r(0)^{1-d}+r(0)\Big) = 0.
\end{equation}
Since we assume that the dynamics of \eqref{eqru} is not stationary, $u(0)$ and $-c_d m (d-2)r(0)^{1-d}+r(0)$ are not simultaneously zero for any characteristic. Therefore, the quantity $\theta$ in \eqref{theta}, now written as
\begin{equation}\label{theta0}
    \theta := \left\{\hspace{-2mm}
    \begin{array}{cl}
\displaystyle  \frac{1-\frac{P(0)r(0)}{d m}}{u(0)}   & \mbox{if } u(0)\neq 0 \\[5mm]
\displaystyle      \frac{\frac{1}{d}P(0)u(0)-m w(0)}{m(-c_d m (d-2)r(0)^{1-d}+r(0))}   & \mbox{if } -c_d m (d-2)r(0)^{1-d}+r(0)\neq 0 
    \end{array}
     \right. 
\end{equation}
in terms of the initial condition of the characteristic ODE system, is always a well-defined real number for every characteristic.

Using \eqref{eqru}, one can directly verify that the solution to \eqref{eqrad_d} is given by
\begin{equation}\label{Pw}
    P(t) = \frac{P(0)}{f(t)},\quad w(t) = \frac{f'(t)}{f(t)}
\end{equation}
where
\begin{equation}\label{Pwf}
    f(t) = \theta u(t) + \frac{P(0)}{d m}r(t).
\end{equation}
The detail of the verification is given in Remark \ref{rem_checkPw} below. Since $f(0)=1$ (see Remark \ref{rem_checkPw}), \eqref{eqrad_d} has a global solution if and only if $f(t)>0$ for any $t\ge 0$. Since $(r(t),u(t))$ travels on the energy curve $\{(\tilde{r},\tilde{u})\in\mathbb{R}_+\times\mathbb{R}:m^{-2/d}(\frac{1}{2}\tilde{u}^2 + m N(\tilde{r}) + \frac{1}{2}\tilde{r}^2)=C_0\}$, the condition that $f(t)>0,\,\forall t\ge 0$ for every characteristic is equivalent to \eqref{thm1_1}.

Therefore, in view of the extension criterion Lemma \ref{lem_crit}, with item 1 of Theorem \ref{thm1} assumed, the existence of global classical bulk solution to \eqref{eqrad} is equivalent to \eqref{thm1_1}.

\end{proof}

\begin{remark}\label{rem_checkPw}
In this remark we verify that \eqref{Pw} solves \eqref{eqrad_d} as long as $f(t)$ keeps positive. We first notice that $f(0)=1$. In fact, if $u(0)\ne 0$, this is clear from the first expression in \eqref{theta0}; otherwise, \eqref{C0d0} gives $1-\frac{P(0)r(0)}{d m}=0$ since the ODE system \eqref{eqru} is assumed to be non-stationary, and $f(0)=1$ also follows.

Then we calculate $f'$ and $f''$ by 
\begin{equation}
    f' = \theta u' + \frac{P(0)}{d m}r' = \theta(c_d m (d-2) r^{1-d} - r) + \frac{P(0)}{d m} u
\end{equation}
and
\begin{equation}\label{ddf}\begin{split}
    f'' = & \theta(-c_d m (d-2)(d-1) r^{-d} - 1)r' + \frac{P(0)}{d m} u'\\  = & \theta(-c_d m (d-2)(d-1) r^{-d} - 1)u + \frac{P(0)}{d m} (c_d m (d-2) r^{1-d} - r)\,.
\end{split}\end{equation}
It follows that $f'(0)=w(0)$ by reasoning similarly as the previous paragraph.

Therefore \eqref{Pw} satisfies the initial condition of \eqref{eqrad_d}, and the $P$-equation in \eqref{eqrad_d} is clear. To check the $w$-equation, we first calculate
\begin{equation}
    w' = \frac{f''}{f} - \frac{(f')^2}{f^2}= \frac{f''}{f} - w^2\,.
\end{equation}
Therefore it suffices to check
\begin{equation}
    -c_d m (d-1)(d-2)r^{-d} + c_d (d-2)r^{1-d} \frac{P(0)}{f} - 1 = \frac{f''}{f}\,,
\end{equation}
i.e.,
\begin{equation}
    (-c_d m (d-1)(d-2)r^{-d}-1)\Big(\theta u + \frac{P(0)}{d m}r\Big) + c_d (d-2)r^{1-d} P(0)  = f''\,.
\end{equation}
This coincide with the previous calculation of $f''$ in \eqref{ddf}.

\end{remark}

\begin{proof}[Proof of Theorem \ref{thm1_2d}]

Similar to the previous proof, we may assume that item 1 of Theorem \ref{thm1_2d} is satisfied, and $C_0 > \min_{\tilde{r}} \{N(\tilde{r})+\frac{1}{2}\tilde{r}^2\}$, so that every characteristic \eqref{eqru} of the solution is non-stationary and has the same period $T_0=T(C_0;1)$. Then it suffices to analyze whether the $(P,w)$ dynamics \eqref{eqrad_d} has a global solution for given $r\in (0,R_0)$.

\eqref{C0d_2d} shows that for every characteristic, the initial condition for the ODE system  \eqref{eqru}\eqref{eqrad_d} for $(r,u,P,w)$ satisfies
\begin{equation}
    u(0)\Big(-\frac{1}{2}P(0)u(0)+m w(0)\Big) + m\Big(1-\frac{P(0)r(0)}{2 m}\Big)\Big(-\frac{1}{2\pi} m r(0)^{-1}+r(0)\Big) = 0.
\end{equation}
Since we assume that the dynamics of \eqref{eqru} is not stationary, $u(0)$ and $-\frac{1}{2\pi}m r(0)^{-1}+r(0)$ are not simultaneously zero for any characteristic. Therefore, the quantity $\theta$ in \eqref{theta_2d}, now written as
\begin{equation}
    \theta := \left\{\hspace{-2mm}
    \begin{array}{cl}
\displaystyle  \frac{1-\frac{P(0)r(0)}{2 m}}{u(0)}   & \mbox{if } u(0)\neq 0 \\[5mm]
\displaystyle      \frac{\frac{1}{2}P(0)u(0)-m w(0)}{m(-\frac{1}{2\pi}m r(0)^{-1}+r(0))}   & \mbox{if } -\frac{1}{2\pi} m r(0)^{-1}+r(0)\neq 0 
    \end{array}
     \right.
\end{equation}
in terms of the initial condition of the characteristic ODE system, is always a well-defined real number for every characteristic. Then one can show that \eqref{Pw} with \eqref{Pwf} again solves \eqref{eqrad_d}, and thus \eqref{eqrad_d} has a global solution if and only if $f(t)>0$ for any $t\ge 0$. Then conclusion is obtained similarly as the previous proof.

\end{proof}

\section{Analysis of the period of general 1D Hamiltonian systems}\label{sec_ana}

This section analyzes how the period of general 1D Hamiltonian systems changes with respect to the energy level of the orbit. This section is independent of the content of other sections.

Let $(x(t),v(t))$ be the solution to a 1D Hamiltonian system
\begin{equation}\label{xv}\left\{\begin{split}
& \dot{x} = v \\
& \dot{v}= - V'(x)
\end{split}\right.\end{equation}
where $V$ is a potential function defined on $(X_1,X_2)\subset \mathbb{R}$, where $-\infty\le X_1<X_2\le \infty$. We assume that $V$ satisfies the following properties:
\begin{itemize}
\item[{\bf (V1)}] $V$ is smooth on $(X_1,X_2)$, $\lim_{x\rightarrow X_1^+}V(x)=\lim_{x\rightarrow X_2^-}V(x)=\infty$. There exists $X_0\in (X_1,X_2)$ such that $V(X_0)=V'(X_0)=0$, $V'(x)<0$ on $(X_1,X_0)$ and $V'(x)>0$ on $(X_0,X_2)$.
\item[{\bf (V2)}] $V''(X_0)>0$.
\item[{\bf (V3)}] $V$ is real analytic on $(X_1,X_2)$.
\end{itemize}
The property {\bf (V3)} clearly implies that for any $R_1,R$ with $X_1<R_1<X_0<R<X_2$, $V(x)$ has a holomorphic extension to a neighborhood of the interval $(R_1,R)$ in the complex plane. The total energy
\begin{equation}
\cE(x,v) = \frac{1}{2}v^2 + V(x)
\end{equation}
is conserved along the solution to \eqref{xv}. Due to {\bf (V1)}, any energy level set for an energy level $E>0$ in the $(x,v)$-plane is a compact and connected simple curve, and  any solution travels periodically along such a level curve. For any $E>0$, the period at energy level $E$ is given by
\begin{equation}\label{TE}
\frac{1}{2}T(E) = \int_{x_1(E)}^{x_2(E)} \frac{1}{\sqrt{2(E - V(x))}}\rd{x}
\end{equation}
where $x_1(E)<x_2(E)$ are determined by the equation $E = V(x)$. These are classical results that can be found in \cite{arnol2013mathematical}.

The main result of this section is the following.
\begin{theorem}\label{thm_ana}
Assume $V$ satisfies {\bf (V1)}-{\bf (V3)}. Then $T(E)$ in \eqref{TE} is real analytic on $(0,\infty)$.
\end{theorem}

\subsection{Proof of Theorem \ref{thm_ana}}

To prove Theorem \ref{thm_ana}, we first decompose $T(E)$ into the left and right contributions
\begin{equation}
\frac{1}{2}T(E) = T_1(E)+T_2(E) := \int_{x_1(E)}^{X_0} \frac{1}{\sqrt{2(E - V(x))}}\rd{x} + \int_{X_0}^{x_2(E)} \frac{1}{\sqrt{2(E - V(x))}}\rd{x}.
\end{equation}
By symmetry, it suffices to prove the real analytic property of $T_2(E)$. Denote the inverse functions of $V$ on $(X_1,X_0)$ and $(X_0,X_2)$ by $U_1$ and $U_2$ respectively. In the expression of $T_2$, we use a change of variable $y=V(x)/E$ to obtain
\begin{equation}\label{T2}
T_2(E) = \int_0^1 \frac{1}{\sqrt{2(E -  E y)}}E U_2'(E y)\rd{y} = \sqrt{\frac{E}{2}} \int_0^1 \frac{1}{\sqrt{1 -  y}}U_2'(E y)\rd{y}.
\end{equation}
This allows us to take derivative with respect to $E$ and get
\begin{equation}\label{dT2}\begin{split}
T_2'(E)  = & \frac{1}{2\sqrt{2}\sqrt{E}}\Big(\int_0^1 \frac{1}{\sqrt{1 -  y}}U_2'(E y)\rd{y} +  2E \int_0^1 \frac{1}{\sqrt{1 -  y}}y U_2''(E y)\rd{y}\Big) \\
= & \frac{1}{2\sqrt{2}\sqrt{E}}\int_0^1 \frac{1}{\sqrt{1 -  y}}\Big(U_2'(E y) +  2E y U_2''(E y)\Big)\rd{y} \\
= & \frac{1}{2\sqrt{2}E^{3/2}}\int_0^E \frac{1}{\sqrt{1 -  y/E}}\Big(U_2'(y) +  2y U_2''(y)\Big)\rd{y}
\end{split}\end{equation}
Notice that the last integral is absolutely convergent near $y=0$ because $U_2(y)-X_0\sim \sqrt{y}$, $U_2'(y)\sim y^{-1/2}$,  $U_2''(y)\sim -y^{-3/2}$. This computation shows that $T(E)$ is at least differentiable.

The proof of Theorem \ref{thm_ana} is based on the formula \eqref{T2}, which can be extended to certain \emph{complex values} of $E$. For this purpose, we need to have a holomorphic extension of $U_2$, which makes the quantity $U_2'(E y)$ well-defined for complex $E$ and $y\in (0,1)$. Then the holomorphic property of $T(E)$ can be easily obtained by showing that \eqref{dT2} is also valid for complex $E$.

The holomorphic extension of $U_2$ has to be constructed very carefully because $U_2(E)$ cannot be extended to negative values of $E$, but the usage of $U_2'(E y)$ does need the value of $U_2'$ on a ray emanating from the origin. The original real function $V$ maps $(X_0,X_2)$ to $(0,\infty)$, and $U_2$ maps $(0,\infty)$ to $(X_0,X_2)$. Therefore, our strategy is to construct the holomorphic extension of $V$  in an angle-shaped region in $\mathbb{C}$ containing the interval $(X_0,X_2)$, and then show that one can invert the extended $V$ and obtain an extension of $U_2$ on a similar angle-shaped region. Such region contains some rays emanating from the origin.

For $\epsilon>0$ small and $0<R<X_2-X_0$, denote the open region
\begin{equation}
D_{\epsilon,R} = \{z\in\mathbb{C}: 0<|z-X_0|<R,\,|\text{Arg}(z-X_0)|<\epsilon\}
\end{equation}
which will serve as the domain of a holomorphic extension of $V$. See Figure \ref{fig2} as illustration.

\begin{figure}
    \centering
    \includegraphics[width=0.99\textwidth]{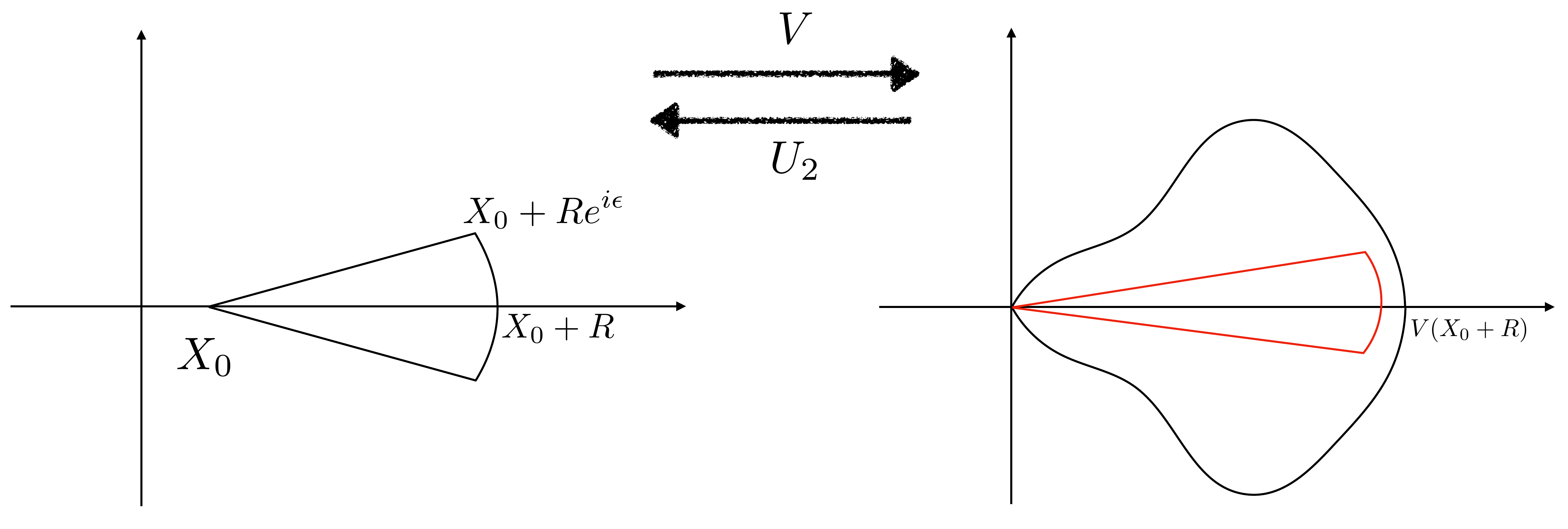}
    \caption{Left: the region $D_{\epsilon,R}$. It is mapped to the black region on the right. Right: the red region is $\tilde{D}_{\tilde{\epsilon},\tilde{R}}$, on which $U_2$ has a holomorphic extension given by Lemma \ref{lem_U21}.}
    \label{fig2}
\end{figure}

\begin{lemma}\label{lem_holo}
Assume {\bf (V1)}-{\bf (V3)}. Then for any $0<R<X_2-X_0$, there exists $\epsilon>0$ such that $V$ has a holomorphic extension to $D_{\epsilon,R}$, satisfying $V'(z)\ne 0$ on $D_{\epsilon,R}$, and one-to-one on $D_{\epsilon,R}$.
\end{lemma}

\begin{proof}
{\bf (V3)} implies that $V$ has a holomorphic extension to $D_{\epsilon,R}\cup B(X_0;\epsilon)$ as long as $\epsilon$ is sufficiently small. Since $V'(X_0)=0$ by {\bf (V1)} and the zeros of the holomorphic function $V'(z)$ are isolated, we see that $V'(z)\ne 0$ for $z\in B(X_0;\epsilon)\backslash \{X_0\}$ as long as $\epsilon$ is sufficiently small. Since $V'(x)>0$ for $x\in [X_0+\epsilon/2,R]$ by  {\bf (V1)}, we see that the same is true in a complex neighborhood of $[X_0+\epsilon/2,R]$. Therefore $V'(z)\ne 0$ in $D_{\epsilon,R}$ up to choosing a smaller $\epsilon$.

{\bf STEP 1}: We first show that $V$ is one-to-one on $D_{\epsilon,R}\cap B(X_0;\epsilon)$ if $\epsilon$ is sufficiently small. In fact, if $z\in B(X_0;\epsilon)$, then by {\bf (V2)} we have 
\begin{equation}
V(z) = \phi(z)(z-X_0)^2,\quad \phi(X_0) =  \frac{1}{2}V''(X_0) > 0
\end{equation}
where $\phi$ is holomorphic, since $V(X_0)=V'(X_0)=0$. Therefore, if $\epsilon$ is sufficiently small, there exists another holomorphic function $\psi$ defined on $B(X_0;\epsilon)$ such that $\psi^2=\phi$ with $\psi(X_0)=\sqrt{\frac{1}{2}V''(X_0)} > 0$.

The map $z\mapsto \psi(z)(z-X_0)$ is clearly one-to-one on $B(X_0;\epsilon)$ by the implicit function theorem. $\psi(B(X_0;\epsilon))$, the image of  $B(X_0;\epsilon)$ under $\psi$, is a neighborhood of the positive real number $\sqrt{\frac{1}{2}V''(X_0)}$ in $\mathbb{C}$. If $\epsilon$ is sufficiently small, then any element  $y\in\psi(B(X_0;\epsilon)\cap D_{\epsilon,R})$ has $|\text{Arg}(y)|\le C\epsilon$ by continuity since the set $\psi(B(X_0;\epsilon)\cap D_{\epsilon,R})$ is bounded away from zero and inside a small neighborhood of $\sqrt{\frac{1}{2}V''(X_0)}$. This implies that the image of $\psi(B(X_0;\epsilon)\cap D_{\epsilon,R})$ under the map $z\mapsto \psi(z)(z-X_0)$ lies in $\{y\in \mathbb{C}: |\text{Arg}(y)|\le C\epsilon\}$. Therefore the map $V$, as the composition of this map with $y\mapsto y^2$, is one-to-one on $D_{\epsilon,R}\cap B(X_0;\epsilon)$ if $\epsilon$ is sufficiently small.

{\bf STEP 2}: We then show that  for any given small $\epsilon_1>0$,  $V$ is one-to-one on $B([X_0+\epsilon_1,X_0+R];\epsilon_2):= \{z\in\mathbb{C}:\dist(z,[X_0+\epsilon_1,X_0+R])<\epsilon_2\}$ for some $0<\epsilon_2=\epsilon_2(\epsilon_1) \le \epsilon_1$.

To see this, we first notice that there exists $c_1>0$ such that $V'(x)\ge c_1$ for any $x\in [X_0+\epsilon_1,X_0+R]$. Therefore, in $B([X_0+\epsilon_1,X_0+R];\epsilon_2)$ for small $\epsilon_2>0$, we have $\Re(V'(z))\ge c_1/2$. For any distinct numbers $z_0,z_1\in B([X_0+\epsilon_1,X_0+R];\epsilon_2)$, we then have
\begin{equation}
V(z_1)-V(z_0) = \int_0^1\frac{\rd}{\rd{t}}V((1-t)z_0+t z_1)\rd{t} = (z_1-z_0)\int_0^1V'((1-t)z_0+t z_1)\rd{t}\ne 0
\end{equation}
since the last integral has a real part at least $c_1/2$. This shows that $V$ is one-to-one on $B([X_0+\epsilon_1,X_0+R];\epsilon_2)$.

\begin{figure}
    \centering
    \includegraphics[width=0.8\textwidth]{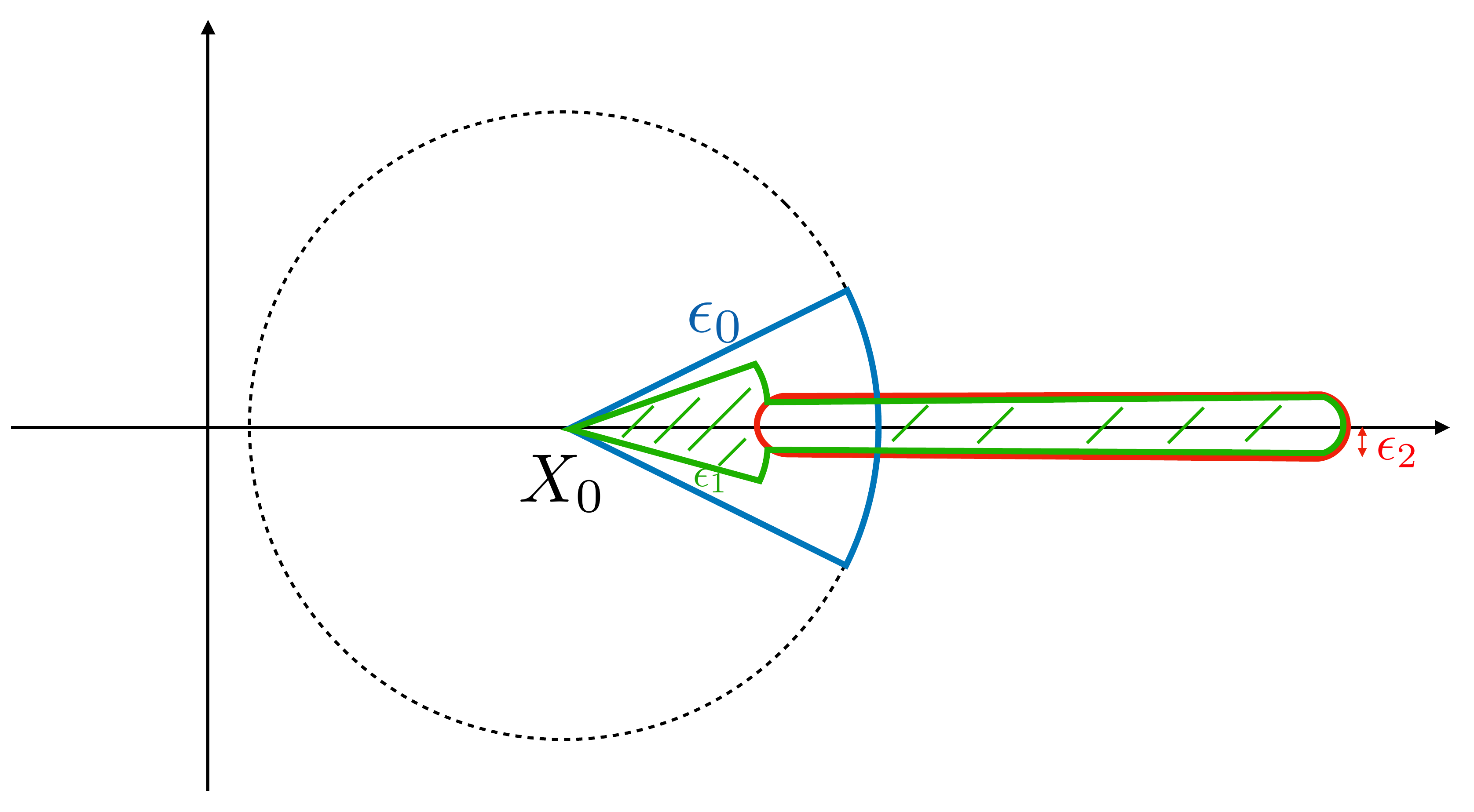}
    \caption{STEP 3 of the proof of Lemma \ref{lem_holo}. The blue region is $D_{\epsilon_0,R}\cap B(X_0;\epsilon_0)$, and the red region is $B([X_0+\epsilon_1,X_0+R];\epsilon_2)$. We know from previous steps that $V$ is one-to-one on either region. To see that it is one-to-one on the green region, it suffices to show that the two shaded regions do not give identical $V$ values.}
    \label{fig3}
\end{figure}

{\bf STEP 3}: Glue the two regions, see Figure \ref{fig3} as illustration. We take $\epsilon_0$ so that $V$ is one-to-one on $D_{\epsilon_0,R}\cap B(X_0;\epsilon_0)$ as in STEP 1, take $\epsilon_1<\epsilon_0$ to be determined, and then choose $\epsilon_2=\epsilon_2(\epsilon_1)$ according to STEP 2. Since $V(x)>0$ for real inputs $x\in [X_0+\epsilon_0,X_0+R]$, we may further require $\epsilon_2$ to be sufficiently small so that $|V(z)|> c$ for any $z\in B([X_0+\epsilon_1,X_0+R];\epsilon_2)\backslash (D_{\epsilon_0,R}\cap B(X_0;\epsilon_0))$, with $c>0$ independent of $\epsilon_1$.

We claim that for sufficiently small $\epsilon_1$, $V$ is one-to-one on $(D_{\epsilon_1,R}\cap B(X_0;\epsilon_1))\cup B([X_0+\epsilon_1,X_0+R];\epsilon_2)$, which would finish the proof. To prove the claim, since we already know that $V$ is one-to-one on $D_{\epsilon_0,R}\cap B(X_0;\epsilon_0)$ and $B([X_0+\epsilon_1,X_0+R];\epsilon_2)$, it suffices to show that for any $z_1\in D_{\epsilon_1,R}\cap B(X_0;\epsilon_1)$ and $z_2\in B([X_0+\epsilon_1,X_0+R];\epsilon_2)\backslash (D_{\epsilon_0,R}\cap B(X_0;\epsilon_0))$ we have $V(z_1)\ne V(z_2)$. This is clear since $|V(z_2)|>c$ by the choice of $\epsilon_2$, but $|z_1-X_0|<\epsilon_1$, and thus $|V(z_1)|=|V(z_1)-V(X_0)| < c/2$ if $\epsilon_1$ is sufficiently small.

\end{proof}

Therefore $V$ is invertible on $D_{\epsilon,R}$ and we denote its inverse as $U_2$, which is also a holomorphic function with non-vanishing derivative. Since $V$ maps the interval $(X_0,X_0+R)$ to $(0,V(X_0+R))$ and behaves like $\frac{V''(X_0)}{2}(z-X_0)^2$ for $z$ near $X_0$, we see that  for any $\epsilon_1>0$, the domain of $U_2$ contains a set of the form $\tilde{D}_{\epsilon_2,V(X_0+R)-\epsilon_1}$ for sufficiently small $\epsilon_2>0$, where
\begin{equation}\label{U2domain}
\tilde{D}_{\tilde{\epsilon},\tilde{R}} := \{y\in\mathbb{C}: 0<|y|<\tilde{R},\,|\text{Arg}(y)|<\tilde{\epsilon}\}
\end{equation}
see Figure \ref{fig2}. Recall that due to {\bf (V1)} we have $V(X_2^-)=\infty$. Since $0<R<X_2-X_0$ can be taken as arbitrarily close to $X_2-X_0$, we obtain the following conclusion. 

\begin{lemma}\label{lem_U21}
For any $\tilde{R}>0$, there exists $\tilde{\epsilon}>0$ such that $U_2$ has a holomorphic extension from the real interval $(0,\tilde{R})$ to $\tilde{D}_{\tilde{\epsilon},\tilde{R}}$.
\end{lemma}

Then we estimate the derivatives of the holomorphic function $U_2$.
\begin{lemma}\label{lem_U22}
Fix $\tilde{R}>0$, and let $\tilde{\epsilon}$ and the holomorphic function $U_2$ be as in the previous lemma. Then $U_2$ satisfies the estimates
\begin{equation}
|U_2'(y)| \le C|y|^{-1/2},\quad |U_2''(y)| \le C|y|^{-3/2}
\end{equation}
for any $y\in \tilde{D}_{\tilde{\epsilon},\tilde{R}}$, with $C$ possibly depending on $\tilde{R}$.
\end{lemma}

\begin{proof}
We may take a region $D_{\epsilon,R}$ on which $V$ is one-to-one and $V(D_{\epsilon,R})$ contains $y\in \tilde{D}_{\tilde{\epsilon},\tilde{R}}$. We may assume that $|y|$ is sufficiently small, which is equivalent to the condition that $z=U_2(y)\in D_{\epsilon,R}$ is sufficiently close to $X_0$. We fix such a $y\in \tilde{D}_{\tilde{\epsilon},\tilde{R}}$, and denote $z=U_2(y)$. Then the differentiation rule for inverse functions gives
\begin{equation}
U_2'(y) = \frac{1}{V'(z)},\quad U_2''(y) = -\frac{V''(z)}{(V'(z))^3}\,.
\end{equation}
By {\bf (V1)}-{\bf (V3)}, we have the convergent power series expansion
\begin{equation}
V(z) = \frac{V''(X_0)}{2}(z-X_0)^2 + \sum_{n=3}^\infty a_n(z-X_0)^n,\quad V''(X_0)>0
\end{equation}
in a neighborhood of $X_0$. Note that $|y| = |V(z)| \le C|z-X_0|^2$ in that neighborhood. Moreover, we have
\begin{equation}
V'(z) = V''(X_0)(z-X_0) + \sum_{n=2}^\infty (n+1)a_{n+1}(z-X_0)^n
\end{equation}
and
\begin{equation}
V''(z) = V''(X_0) + \sum_{n=1}^\infty (n+1)(n+2)a_{n+2}(z-X_0)^n
\end{equation}
from which we conclude that $V''(z)$ is bounded near $X_0$, while $|V'(z)| \ge c|z-X_0|$ near $X_0$ since $V''(X_0)>0$.  Therefore we obtain the conclusion.

\end{proof}

\begin{proof}[Proof of Theorem \ref{thm_ana}]
Lemma \ref{lem_U21} gives the holomorphic extension of $U_2$ to a region of the form $\tilde{D}_{\tilde{\epsilon},\tilde{R}}$. Then for any complex number $E\in \tilde{D}_{\tilde{\epsilon},\tilde{R}}$, we may define $T_2(E)$ by the RHS of \eqref{T2} since the input $E y$ for $U_2'$ always lies in $\tilde{D}_{\tilde{\epsilon},\tilde{R}}$ for $y\in (0,1)$, and the integral converges due to Lemma \ref{lem_U22}. This extends the original definition of $T_2(E)$ for $E\in (0,\tilde{R})$, and it is holomorphic because one can take $E$-derivative by \eqref{dT2} due to Lemma \ref{lem_U22}. This shows that the extended $T_2(E)$ is holomorphic in $\tilde{D}_{\tilde{\epsilon},\tilde{R}}$. In particular, $T_2(E)$ is real-analytic on $(0,\tilde{R})$. Since $\tilde{R}$ in Lemma \ref{lem_U21} can be taken arbitrarily large, we see that $T_2(E)$ is real-analytic on $(0,\infty)$. Since $T_1(E)$ can be treated similarly, we see that $T(E)$ is real-analytic on $(0,\infty)$.
\end{proof}

\subsection{Local expansion for $T(E)$ near $E=0$}

In this subsection we give a sufficient condition which guarantees that $T(E)$ is non-constant. This is based a local expansion for $T(E)$ near $E=0$, combined with Theorem \ref{thm_ana}. Then we apply this theory to the potential arising from the Euler-Poisson system.

\begin{lemma}\label{lem_Texpan}
Assume $V$ satisfies {\bf (V1)}-{\bf (V3)}. Then for small $E>0$ we have
\begin{equation}\label{lem_Texpan_1}
T'(E) = \frac{\pi c_V}{(V''(X_0))^{7/2}} + O(E^{1/2}),
\end{equation}
where 
\begin{equation}
c_V := -\frac{1}{4}V''(X_0)V''''(X_0) + \frac{5}{12}(V'''(X_0))^2 .
\end{equation}
As a consequence, if $c_V\ne 0$, then $T(E)$ is non-constant on any sub-interval of $(0,\infty)$.
\end{lemma}

\begin{proof}
We start from \eqref{dT2}. Denoting $x_2=U_2(y)$ for $y>0$, we have
\begin{equation}
U_2'(y) = \frac{1}{V'(x_2)},\quad U_2''(y) = -\frac{V''(x_2)}{(V'(x_2))^3}
\end{equation}
by implicit differentiation. Therefore the last integrand in \eqref{dT2} is 
\begin{equation}\begin{split}
U_2'(y) +  2y U_2''(y) = \frac{1}{V'(x_2)} - \frac{2V(x_2)V''(x_2)}{(V'(x_2))^3} = \frac{(V'(x_2))^2-2V(x_2)V''(x_2)}{(V'(x_2))^3}:=H(x_2).
\end{split}\end{equation}
Using a reflection about $X_0$, we get a similar formula for $T_1$ as
\begin{equation}\begin{split}
T_1'(E)  =  & \frac{1}{2\sqrt{2}E^{3/2}}\int_0^E \frac{1}{\sqrt{1 -  y/E}}\Big(-U_1'(y) +  2y U_1''(y)\Big)\rd{y}
\end{split}\end{equation}
with
\begin{equation}\begin{split}
-U_1'(y) +  2y U_1''(y) = \frac{(V'(x_1))^2-2V(x_1)V''(x_1)}{-(V'(x_1))^3}=-H(x_1)
\end{split}\end{equation}
where $x_1=U_1(y)$, $U_1$ being the inverse function of $V$ on $(X_1,X_0)$. Here there is an extra negative sign on the denominator because the reflection changes the sign of $V'$. Therefore, since $T'(E)  =  2(T_1'(E)+T_2'(E))$, we obtain the formula
\begin{equation}\label{dTE}\begin{split}
T'(E)  =  \frac{1}{\sqrt{2}E^{3/2}}\int_0^E \frac{1}{\sqrt{1 -  y/E}}\Big(H(x_2)-H(x_1)\Big)\rd{y},
\end{split}\end{equation}
where $x_1<X_0<x_2$ are determined by $V(x_1)=V(x_2)=y$.

Write the Taylor expansion of $V(x)$ near $X_0$ as
\begin{equation}
V(x) = a_2(x-X_0)^2 + a_3(x-X_0)^3 + a_4(x-X_0)^4 + O(|x-X_0|^5),\quad a_2>0
\end{equation}
where $a_k = \frac{V^{(k)}(X_0)}{k!},\,k=2,3,4$. Explicit calculation shows that
\begin{equation}
H(x) = \frac{1}{8 a_2^3}(-4a_2a_3 + (-12a_2a_4+15a_3^2)(x-X_0)) + O(|x-X_0|^2).
\end{equation}
Also, if we take $x_1<X_0<x_2$ with $V(x_1)=V(x_2)=y>0$ small, then 
\begin{equation}
x_{1,2} = X_0 \pm a_2^{-1/2}y^{1/2} + O(y).
\end{equation}
Therefore the integrand in \eqref{dTE} is
\begin{equation}\begin{split}
H(x_2)-H(x_1)
= \frac{1}{8 a_2^3}(-12a_2a_4+15a_3^2)2a_2^{-1/2}y^{1/2} + O(y) = \frac{c_V}{4 a_2^{7/2}} y^{1/2} + O(y).
\end{split}\end{equation}
Therefore the integral on the RHS of \eqref{dTE} is
\begin{equation}\begin{split}
\frac{c_V}{4 a_2^{7/2}}\int_0^E \frac{1}{\sqrt{1 -  y/E}}y^{1/2}\rd{y} + O\Big(\int_0^E \frac{1}{\sqrt{1 -  y/E}}y\rd{y} \Big) = \frac{\pi c_V}{8 a_2^{7/2}} E^{3/2} + O(E^2)
\end{split}\end{equation}
and \eqref{lem_Texpan_1} is proved.

Assume $c_V\ne 0$. Theorem \ref{thm_ana} shows that $T'(E)$ is real-analytic on $(0,\infty)$, and \eqref{lem_Texpan_1} shows that $T'(E)$ is nonzero near $E=0$. Therefore $T'(E)$ is not identically zero on any sub-interval of $(0,\infty)$, i.e., $T(E)$ is non-constant on any sub-interval of $(0,\infty)$.
\end{proof}

\begin{corollary}\label{cor_Vd}
Let $d\ge 2$ be an integer, and 
\begin{equation}\label{Vd}
V(x) = \frac{x^{2-d}-1}{d-2} + \frac{x^2-1}{2}
\end{equation}
where the first term is replaced by $-\ln x$ for $d=2$. Then the corresponding period $T(E)$ is non-constant on any sub-interval of $(0,\infty)$ if and only if $d\ne 4$. Furthermore, $T(E)$ is constant if $d=4$.
\end{corollary}

By rescaling arguments and calculating \eqref{Tmin}, one can easily deduce that for $d=4$, $T(\cE_0;1)$ takes the constant value $\pi$, as observed in Figure \ref{fig0}.

\begin{proof}
It is clear that {\bf (V1)}-{\bf (V3)} are satisfied with $(X_1,X_0,X_2)=(0,1,\infty)$. Then we compute
\begin{equation}
V''(x) = (d-1)x^{-d} + 1,\quad V'''(x) = -d(d-1)x^{-d-1},\quad V''''(x) = (d+1)d(d-1)x^{-d-2}\,.
\end{equation}
Therefore
\begin{equation}
c_V = -\frac{1}{4}d^2(d+1)(d-1) + \frac{5}{12}d^2(d-1)^2 = \frac{1}{12}d^2(d-1)(-3d-3+5d-5)= \frac{1}{6}d^2(d-1)(d-4)
\end{equation}
which is nonzero for any integer $d\ge 2,\,d\ne 4$. Therefore we get the conclusion for $d\ne 4$ from Lemma \ref{lem_Texpan}.

For $d=4$, we have $V(x) = \frac{1}{2x^2}+\frac{x^2}{2}-1$. We follow the notation in \eqref{dTE}. Since $x_1,x_2$ are determined by $V(x_1)=V(x_2)=y$, we have $x_2=x_1^{-1}$. Then notice that $V'(x) = -\frac{1}{x^3}+x$, $V''(x) = \frac{3}{x^4}+1$. Then explicit calculation (substituting $x_2=x_1^{-1}$) shows that
\begin{equation}\begin{split}
H(x_2)-H(x_1)= -\frac{(V'(x_1))^2-2V(x_1)V''(x_1)}{(V'(x_1))^3}+ \frac{(V'(x_2))^2-2V(x_2)V''(x_2)}{(V'(x_2))^3} =0\,.
\end{split}\end{equation}
Therefore $T(E)$ is constant due to \eqref{dTE}.
\end{proof}

\begin{remark}
It is well-known that $T(E)$ is constant for \eqref{Vd} with $d=1$ because $V$ is a sort of translated harmonic oscillator. This is consistent with the well-known theory of 1D Euler-Poisson \cite[Theorem 3.1]{engelberg2001critical}. However, we are not aware of a previous result which noticed the special property of the case $d=4$.
\end{remark}

\section*{Acknowledgements}
JAC and RS were supported by the Advanced Grant Nonlocal-CPD (Nonlocal PDEs for Complex Particle Dynamics: Phase Transitions, Patterns and Synchronization) of the European Research Council Executive Agency (ERC) under the European Union's Horizon 2020 research and innovation programme (grant agreement No. 883363). JAC was also partially supported by the EPSRC grant numbers EP/T022132/1 and EP/V051121/1.

\bibliographystyle{abbrv}
\bibliography{biblio}

\end{document}